\documentclass{amsart}
\usepackage{amssymb,latexsym,amscd, amsthm, epsfig,amsmath,amssymb}
\usepackage{color}

 \newtheorem{thm}{Theorem}[section]
 \newtheorem{cor}[thm]{Corollary}

 \newtheorem{lem}[thm]{Lemma}
 \newtheorem{prop}[thm]{Proposition}
 \theoremstyle{definition}
 \newtheorem{defn}[thm]{Definition}
 \theoremstyle{remark}
 \newtheorem{rem}[thm]{Remark}
 \theoremstyle{definition}
 \newtheorem{ex}[thm]{Example}

\newtheorem*{conj*}{Conjecture}
 \newtheorem*{cor*}{Corollary}

 \newcommand{\X}{\Bbb{X}}

 \newcommand{\PP}{\mathbb{P}}

\def\move-in{\parshape=1.75true in 5true in}

\newcommand\rk{\mathrm{rk}}

\definecolor{Orange}{cmyk}{0, 0.7, 0.7, 0.1}

\usepackage{xypic}

\begin{document}

\title{{ Symmetric tensors: rank, Strassen's conjecture and $\lowercase{e}$-computability}}

\author[E. Carlini]{Enrico Carlini}
\address[E. Carlini]{Department of Mathematical Sciences, Politecnico di Torino, Turin, Italy}
\email{enrico.carlini@polito.it}
\address[E.Carlini]{School of Mathematical Sciences, Monash University, Clayton, Australia.}
\email{enrico.carlini@monash.edu}

\author[M.V.Catalisano]{Maria Virginia Catalisano}
\address[M.V.Catalisano]{Dipartimento di Ingegneria Meccanica, Energetica, Gestionale e dei Trasporti, Universit\`{a} degli studi di Genova, Genoa, Italy.}
\email{catalisano@diptem.unige.it}

\author[L. Chiantini]{Luca Chiantini}
\address[L. Chiantini]{Department of Information Engineering and Mathematics, University of Siena,Italy.}
\email{luca.chiantini@unisi.it}

\author[A.V. Geramita]{Anthony V. Geramita}
\address[A.V. Geramita]{Department of Mathematics and Statistics, Queen's University, King\-ston, Ontario, Canada and Dipartimento di Matematica, Universit\`{a} degli studi di Genova, Genoa, Italy}
\email{Anthony.Geramita@gmail.com \\ geramita@dima.unige.it  }

\author[Y. Woo]{Youngho Woo}
\address[Y. Woo]{National Institute of Mathematical Sciences, Daejeon, South Korea}
\email{youngw@nims.re.kr}


%%%%%%%%%%%%%%%%%%%%

\begin{abstract}
 In this paper we introduce a new method to produce lower bounds for the Waring rank of symmetric tensors. We also introduce the notion of {$e$-computability} and we use it to prove that Strassen's Conjecture holds in infinitely many new cases.
\end{abstract}

%%%%%%%%%%%%%%%%%%%%

\maketitle

% \today

%%%%%%%%%%%%%%%%%%%%

\section{Introduction}

Let $k$ be a  field of characteristic zero and let $F \in k[x_0,x_1,\ldots , x_n]= S = \oplus S_i\  (i\geq 0 $ and $ n \geq 1$) be a homogeneous polynomial ({\it form}) of degree $d$ i.e. $F \in S_d$.  It is well known that in this case  each $S_i$ has a basis consisting of $i^{th}$ powers of linear forms.  Thus we may write
$$
F = \sum_{i=1}^r \alpha_i L_i^d \hskip 2cm \alpha_i \in k, L_i \in S_1.
$$
If $k$ is algebraically closed (which we now assume for the rest of the paper) then each $\alpha_i = \beta_i^d$ for some $\beta_i \in k$ and so we can write
\begin{eqnarray}  \label{WD}
F = \sum_{i=1}^r(\beta_iL_i)^d = \sum_{i=1}^r\tilde{L}_i^d
\end{eqnarray}

We call a description of $F$ as in (\ref {WD}), a {\it Waring Decomposition of $F$}.  The least integer $r$ such that $F$ has a Waring Decomposition with exactly $r$ summands is called the {\it Waring Rank} (or simply the {\it rank}) of $F$.

There are several variants on this notion in the literature (see e.g.\cite{RS00}, \cite{landsbergbook}, \cite{BBM:1}). In this paper we will only be interested in the notion of rank described above.

It is easy to see that $F$ has rank one if and only if $[F]\in \PP(S_d)$ is
a point of the  {\it Veronese variety}, $\mathbb{V} \subset \PP(S_d)$.  If $F$ has rank $r$
then $[F] \in \PP(S_d)$ is on $\sigma_r(\mathbb{V})$, the $(r-1)^{st}$ {\it secant variety} of $\mathbb{V}$.

Given a Waring Decomposition of $F$
$$
F=L_1^d + \ldots  + L_\ell^d \ \hbox{ with } L_i = a_{i0}x_0 + \ldots + a_{in}x_n,
$$
we can associate a set of $\ell$ points in $\PP^n$ to this decomposition, namely
$$
\X= \{ [a_{10}:\ldots :a_{1n}], \ldots , [a_{\ell 0}: \ldots :a_{\ell n}] \}.
$$
The importance of this set will be explained a bit further on.

Let $T = k[X_0, \ldots , X_n] = \oplus T_i  \ (i\geq 0)$ be another polynomial ring and let $T$ act on $S$ by
setting
$$
X_i \circ F = (\partial/\partial x_i) (F)
$$
and extending linearly (see  \cite{Ge} or \cite{IaKa}
).  With this action we write
$$
F^\perp = \{ g \in T \ \mid \ g\circ F = 0 \}.
$$

If $F$ is a form of degree $d$, then every form in $T$ of degree $\geq d+1$ is in $F^\perp$ and so $F^\perp$ is an Artinian ideal of $T$.  It is a classical theorem of Macaulay that $T/F^\perp$ is also a Gorenstein ring with socle in degree $d$.  Moreover, every Gorenstein Artinian quotient of $T$ with socle in degree $d$ is of the form $T/F^\perp$, with $F$ a form of degree $d$.

Suppose that $F = L^d$ where $L=a_0x_0 + \ldots +a_nx_n$ and $g \in T_\delta$.  Then
$$
g\circ L^d = (d!/\delta !)g(a_0, \ldots , a_n)L^{d-\delta} .
$$
 It follows that if $F \in S_d$ has a  Waring Decomposition
$$
F = L_1^d+ \ldots + L_\ell^d \hbox{ where } L_i\leftrightarrow p_i \in \PP^n \hbox{ and } \Bbb Y =\{p_1, \ldots p_\ell \}
$$
then  for all $g \in T$ such that $g(p_i)=0, i=1, \ldots ,\ell$, $ g \in F^\perp $,  that is
$$
I_{\Bbb Y} \subset F^\perp
$$
where $I_{\Bbb Y}\subset T$ is the ideal of the set $\mathbb{Y}$.

The opposite implication is also true, namely $I_{\Bbb Y} \subset F^\perp$, with $\Bbb Y$ a finite set of $\ell$ points in $\PP^n$, then
$F=L_1^d + \ldots + L_\ell^d$, where the $L_i$ correspond to the points in $\Bbb Y$, as described above.

These containments are referred to as the {\it Apolarity Lemma} and one can find proofs in \cite{IaKa,RS00}.

Having a particular Waring Decomposition of $F$, or equivalently the ideal of a set of distinct points in $F^\perp$, will thus give us upper bounds for the rank of $F$.  We also need some good lower bounds for the rank of $F$.  The importance of finding such lower bounds was underscored in the papers of \cite{LandsbergTeitler2010} and in further work \cite{Teitler2014}.  In \cite{LandsbergTeitler2010}, generalizing a result of  Sylvester, a lower bound was found in terms of ranks of catalecticant matrices and dimensions of the singularity loci in the spaces defined by varieties coming from catalecticant ideals. Our Theorem \ref{firstthm} finds new lower bounds in terms of different invariants of $F$.

Our new approach to the study of the rank is particularly effective in the direction of Strassen's Conjecture. This famous conjecture was stated in the 1973 paper \cite{strassenconj} and is still open (for some recent progress see \cite{CCC14}). The symmetric version of Strassen's Conjecture can be stated as follows: the rank is additive on the sum of forms in different set of variable, that is
\[\mathrm{rk} (F_1+\ldots+F_m)=\mathrm{rk} (F_1)+\ldots+\mathrm{rk} (F_m)\]
if the forms $F_i$ are in distinct sets of variables. In \cite{carcatgermonomi} it was proved  that the conjecture holds if the forms $F_i$ are monomials. In Theorem \ref{coroll} we find several other families of summands for which Strassens's Conjecture is true.

The paper is organized in the following way. In Section \ref{basic}  we recall some of the basic ideas we will use. In Section \ref{lower} we introduce the notion of $e$-computability and use it to establish our new lower bound for the rank  of $F$. In Section \ref{ecomputabilita} we find several infinite families of forms which are $e$-computable and thus compute their rank. In Sections \ref{mainresultsececomp} and \ref{strassen} we show how useful the notion of $e$-computability is in dealing with Strassens's conjecture by giving many new examples of families of forms for which Strassens's conjecture is true.
In Section \ref{examples} we give an example of an infinite family of forms whose rank is computable by ad hoc methods. We show that the first member of this family is not $1$-computable.

%%%%%%%%%%%%%%%%%%%%
\section{Basic facts} \label{basic}
%%%%%%%%%%%%%%%%%%%%

Let
\[S=k[x_0 , \ldots  ,x_n] \mbox{ and } \hskip 2cm  T=k[X_0 ,\ldots, X_n ],\]
 where $k$ is an algebraically closed field of characteristic
zero.
We let $T$ act via differentiation on $S$  as above.

Given a homogeneous ideal $I\subseteq T$ we denote by $$HF(T/I,i)=
\dim_kT_i -\dim_k I_i$$ the {\it Hilbert function} of $T/I$ in degree $i$.
It is well known that the function $HF(T/I,i)$
is eventually a polynomial function with rational coefficients, and this polynomial is called the
{\it Hilbert polynomial} of $T/I$.  We say that an ideal $I\subseteq
T$ is {\it one dimensional} if the Krull dimension of $T/I$ is one,
equivalently the Hilbert polynomial of $T/I$ is some integer
constant, say $\ell$.  In the case that $I \subseteq T$ is one dimensional, then this eventually constant value of the Hilbert Function of $T/I$ is called the {\it
multiplicity} of $T/I$.  If, in addition, $I$ is a radical ideal,
then $I$ is the ideal of a set of $\ell$ distinct points in $\PP ^n$.  We will
use the fact that if $I$  is a saturated ideal and $T/I$ is one dimensional  of
multiplicity $\ell$, then $HF(T/I, i)$ is always at most $\ell$.

Our main tool is the {\it Apolarity Lemma}, the proof of which can
be found in \cite[Lemma 1.31]{IaKa}.

\begin{lem}\label{apolarityLEMMA} (Apolarity Lemma)
Let $\mathbb{X}=\{[L_1],\ldots,[L_\ell]\} \subset \PP(S_1)$ be a set of $\ell$ distinct points, corresponding to the linear forms $L_1,\ldots,L_\ell\in S_1$. If $F\in S_d$, then
\[F=c_1L_1^d+\ldots+c_\ell L_\ell^d,\]
for $c_1,\ldots,c_\ell\in k$, if and only if
\[I_\mathbb{X}\subset F^\perp.\]
\end{lem}

Note that the coefficients $c_i$ are necessary even if $k$ is algebraically closed since some of them could be zero; this is not a minimal decomposition. With the Apolarity Lemma in mind, we make the following definition.

\begin {defn}
a) If $F$ is a form in $S$ and  $\mathbb X \subset \PP^n$ is a set of reduced points for which $I_\mathbb X \subset F^\perp$, then we say that  $\mathbb X$ is {\it apolar} to $F$.

b) If $\mathbb X$ is apolar to $F$ and $| \mathbb X | \leq | \mathbb Y | $ for any other $\mathbb Y$ apolar to $F$, then we say that
$\X $  {\it minimally decomposes} $F$.
\end {defn}

We conclude with the following trivial, but useful, remark  (see Remark 2.3 of \cite {carcatgermonomi}).

\begin{rem}\label{leastvarREM}
The computation of the rank of  $F$ is independent of the polynomial ring in which we consider $F$.

More precisely,
consider a rank $r$ form $F\in k[x_0,\ldots,x_n]$.
Then $F$  has rank $r$   also if we consider $F$ as a form in
$k[x_0,\ldots,x_n,x_{n+1}, \ldots, x_{n+t}]$.
\end{rem}

\section{Lower bound for rank} \label{lower}

It is useful to recall the following well known results.

 \begin {rem} \label{sommahilb}
 Let $J\subset {T}$ be the ideal of a zero-dimensional scheme
and $t\in T_e$ a   homogeneous differentiation  of degree $e$. If $t$ is
not a zero divisor in $T/J$, then from the  exact sequences
\begin{eqnarray}  \label{sequenza}
0 \longrightarrow  (T/J)_{i-e}
\stackrel{ \cdot t} {\longrightarrow} (T/J )_i\longrightarrow
( T/ J+(t))_i \longrightarrow 0,
\end{eqnarray}
we get, for $s \gg 0$,
\begin{eqnarray}  \label{HF}
e \cdot HF(T/J,s)=\sum_{i=0}^s
HF(T/(J+(t)),i).
\end{eqnarray}

\end{rem}

\begin{lem}\label{colonperp}
Let $F(x_0,\ldots,x_n) \in S_d$, then
$$F^\perp : X_i = (X_i \circ F)^\perp
$$
\end{lem}
\begin{proof}
Let $g \in T$ and suppose that we have $g \in F^\perp : X_i$. Now
\[g \in F^\perp : X_i \Longleftrightarrow
(g X_i ) \circ F = 0 \Longleftrightarrow
g \circ (X_i  \circ F) = 0 \Longleftrightarrow
g \in (X_i  \circ F) ^\perp,
\]
and the conclusion follows.
\end{proof}

We are now ready to state and prove our first theorem.

\begin{thm}\label{firstthm}
Let $F \in S_d$ and let $\X \subset  \PP  (T_1)$ be apolar to $F$ (so
 $I_\X \subset F^\perp$).   Let $I \subset T$ be any ideal generated in degree $e>0$  and let $ t \in I_e$. If
 $t$ is not a zero divisor in $T / (I_\X : I)$, then   for $s \gg 0$ we have
\[ e \cdot |\X| \geq \sum _{i=0}^s HF (T/(I_\X : I+(t)) , i)  \geq \sum _{i=0}^{s} HF (T/(F^\perp : I+( t) ), i) .
\]

\end{thm}
\begin {proof}
Note that $ t $ is a non-zero divisor in $T/ (I_\X : I)$ and that $ I_\X : I$ is the saturated ideal of $\mathbb{Y}\subseteq\X$ consisting of all points of  $\X$ not lying on the zero set of $I$.
Thus, by Remark \ref {sommahilb}, we have
\[
{\frac{1} e} \cdot \sum_{i=0}^s HF(T/(I_\X:I + (t)),i)= |\mathbb{Y}|
\]
for $s \gg 0$.
Moreover  for any $s$,
\[\sum_{i=0}^s HF(T/(I_\X:I+ (t)),i) \geq
\sum_{i=0}^s HF(T/(F^\perp:I+ (t)),i) ,\]
since $I_\X $ is contained in $F^\perp$,
 and so we are done .

\end{proof}

 The following corollary gives  a useful lower bound for the rank of $F$.

\begin {cor} \label{maincor}
Let $F \in S_d$.  Let $I \subset T$ be any ideal generated in degree $e>0$ and let $ t $ be a general form in $ I_e$. For $s \gg 0$ we have

\[ \mathrm{rk} (F) \geq \left ({1\over e} \right)
\sum_{i=0}^s HF(T/(F^\perp:(t)+ (t)),i) .\]

\end{cor}

\begin{proof}
Let $\X$ minimally decompose $F$, so $|\X| = \rk (F)$. Since $ t \in I_e$ is a general form, then $t$ is not a zero divisor in $T / I_\X : I$. So by Theorem \ref{firstthm} we are done.
\end{proof}

Notice that the summation on the right side cannot decrease as $s$ increases  and, indeed, the summands are all zero for $s$ big enough. Hence we often use the corollary above with $s= \infty$.
%%%%%%%%%%%%%%%%%%%%%%%%%%%%%%

\begin {defn} \label{lindef}

Let $F \in S_d$ and $e>0$ be an integer.  We  say that $F$ {\it is $e$-computable} if  there exists an
 ideal $I \subset T$  generated in degree $e$  such that for general  $ t \in I_e$ we have
 \[ \rk (F) = \left ({1\over e} \right) \sum_{i=0}^\infty HF(T/(F^\perp: I +(t)),i) .\]
In this case we say that the {\it rank of F is computed by $I$ and $t$}. In case $I=(t) $,
we simply say that the {\it rank of F is computed by  $t$ }

\end {defn}

 \begin {prop}\label {prop1ecomp}

Let $F \in S_d$ and assume that   $\rk (F)$  is computed by $I$ and $t$.  If $ \Bbb X$ minimally decomposes $F$
and if we let $I_{\X' }= I_ \X :I $, then $\X=\X '$
and $I_{ \X } +(t)= F^\perp +(t)$.

\end {prop}

\begin{proof} Since $\rk (F) >0$, then $ I_ \X :I \neq T$ and, since $t$ is general,  we may assume  that $t$ is a non-zero divisor in $T / I_\X : I$.
  By (\ref{sequenza})  we get
  $$|\X' |= \left ({1\over e} \right) \sum_{i=0}^\infty  HF ( T/ (I_ \X : I+(t)) , i ).$$
 Hence we have
$$  \rk(F) = |\X|\geq  |\X'|= \left ({1\over e} \right)  \sum_{i=0}^\infty HF ( T/ (I_ \X : I +(t)) , i ) $$

$$\geq  \left ({1\over e} \right)  \sum_{i=0}^s  HF (T/ (F^\perp  : I + (t)),i ) = \rk(F).
$$
It follows that $\X=\X '$ and $I_ \Bbb X : I +(t)= F^\perp  : I + (t)$.  Hence

$$ F^\perp   + (t) \subseteq   F^\perp  : I + (t)= I_ \X : I +(t)= I_ \Bbb {X'} +(t)
\subset I_ \X +(t) \subseteq F^\perp   + (t),$$
and the conclusion follows.

\end{proof}

%%%%%%%%%%%%%%%%%%%%%%%%%%%%%%

\section{Forms which are $e$-Computable}\label{ecomputabilita}

In this section we give several examples of forms which are $e$-computable for various values of $e$.

 We start by considering forms in two variables, that is $F\in S =k[x_0,x_1]$, and we recall Sylvester's algorithm to compute the rank of $F$, see \cite{comasseiguer}. Since
 $F^\perp$ is a Gorenstein artinian ideal and $F^\perp\subset T=k[X_0,X_1]$, we have that
 $$
 F^\perp=(h_1, h_2)
 $$
 where $ \deg h_1 = d_1 \leq \deg h_2=d_2 $ and $d_1 + d_2 = \deg F+2$ , $h_1$ and $h_2$ having no common factor. If $h_1$ is square free then $\rk(F)=d_1$, otherwise $\rk(F)=d_2$.

\begin {prop} \label{2variab}
If $F\in S =k[x_0,x_1]$ and $F^\perp= (h_1,h_2)$ as above, then

(i) if $h_1$  is not square free and $ h_1=t^2 \widetilde h_1$, then $F$ is $e$-computable, where $e = \deg t$;

(ii) if $h_1$ is square free and $d_1 < d_2$, then  $F$ is $e$-computable for $e \leq { {d_2 - d_1 + 1} \over 2}$;

(iii) if $d_1 = d_2$ we can assume  we are in case (i).

\end{prop}

\begin {proof}

(i)  $h_1$ is not square free, so $\rk(F)=d_2$.

Since in  this case,  $h_1=t^2 \widetilde h_1$, it is easy to see that  $F^\perp:(t) = (t
\widetilde h_1,h_2)$.
It follows that $F^\perp:(t)+(t)=(t,h_2)$.
Noting that $(t,h_2)$  is a complete intersection of degree $e \cdot d_2$, we have $\sum_{i=0}^\infty  HF (T/ (F^\perp  : (t) + (t)),i ) = e \cdot d_2 = e\cdot \rk(F),$
and this completes the proof of (i).

(ii)  $h_1$  is square free  and $  d_1 < d_2$, so $\rk(F)=d_1$.

Let $t$ be  a  form of degree $e \leq { {d_2 - d_1 + 1} \over 2}$ such that $t | h_2$. We claim that
$$F^\perp:(t)+(t)=(t,h_1).$$
It is easy to show that  $F^\perp:(t)=(h_1,h_2/t),$ hence
$F^\perp:(t)+(t)=(t,h_1, h_2/t)$. But $(t,h_1) $ contains all forms of degree at least $e+d_1 -1$ , and $\deg h_2/t = d_2-e \geq e+ d_1 -1$.
Thus $(t,h_1,h_2/t) = (t,h_1) $, and we have proved the claim.
 Hence, $$\sum_{i=0}^\infty  HF (T/ (F^\perp  : (t) + (t)),i ) = e \cdot d_1 = e \cdot  \rk(F).$$

(iii)  If  $d_1=d_2$ then, using the discriminant of a general combination of $h_1$ and $h_2$,  we can  assume that $h_1$ is not square free.
\end{proof}

We now consider monomials in $S= k[x_0, \ldots, x_n]$.
It is shown in \cite {carcatgermonomi} that any monomial is 1-computable. In the next  proposition we generalize this fact.

\begin {prop}\label{monomi} Let $F = x_0^{a_0} x_1^{a_1} \cdots x_n^{a_n}$ where $0 <a_0 \leq a_1 \leq \ldots \leq a_n$.
Then $F$ is $e$-computable for
$$1 \leq e \leq { {a_0 +1 }\over 2 }.$$
\end{prop}
\begin {proof}
We know that  $\rk (F) = \Pi _{i=1}^n (a_i+1)$  (see \cite  {carcatgermonomi}). Now
\[F^\perp : (X_0^e) +  (X_0^e) = (x_0^{a_0-e} x_1^{a_1} \cdots x_n^{a_n})^\perp + (X_0^e)
=(X_0^{a_0-e+1},  X_1^{a_1+1},  \ldots, X_n^{a_n+1}, X_0^e)
\]
\[
=( X_1^{a_1+1},  \ldots, X_n^{a_n+1}, X_0^e).
\]
Hence
\[\sum_{i=0}^\infty HF(T/(F^\perp: (X_0^e) +(X_0^e)),i)= e \cdot  \Pi _{i=1}^n (a_i+1)= e \cdot \rk(F).
\]
\end {proof}

\begin {rem} \label{ecompxmonomi-formebin} It would be interesting to know if the forms of Propositions \ref{2variab} and \ref{monomi} are $e$-computable for $e$'s different from those described in the two propositions.
\end{rem}

In the following propositions we exhibit several other families of $e$-computable forms.

Consider
\[F=x_0^a(x_1^b + \ldots + x_n^b).\]
Since, both for $n=1$ and, by a change of coordinates, for $b=1$,  $F$ is a monomial, we skip those known  cases (see\cite{carcatgermonomi}).

\begin {prop} \label{a+1geqb}

Let $ b \geq2$, $n \geq 2$ and let
\[F = x_0^a(x_1^b + \ldots + x_n^b) \in S =k[x_0, \ldots , x_n].
\]
If $a+1 \geq b$, then $F$ is 1-computable, the rank of $F$ is computed by $I = (X_1, \ldots , X_n)  $ and  a general linear form $t\in I$, and we have

 \[ \rk (F) = (a+1)n.\]

\end{prop}

\begin {proof}
Consider the ideal $I= (X_1, \ldots, X_n) \subset T$. We first calculate
$F^\perp : I$.

\[ F^\perp : I =(F^\perp :  (X_1, \ldots, X_n) )= (F^\perp :  (X_1)) \cap \cdots \cap (F^\perp :  (X_n)) .\]
Thus, by Lemma \ref {colonperp},
\[ F^\perp:I= (x_0^a x_1^{b-1}) ^\perp \cap \cdots \cap(x_0^a x_n^{b-1}) ^\perp
\]
\[ = (X_0 ^{a+1},X_1^b,X_2,\ldots,X_n)\cap \cdots \cap
(X_0 ^{a+1},X_1,\dots,X_{n-1},X_n^b)
\]
\[ =(X_0 ^{a+1},X_1^b,\ldots,X_n^b, X_1X_2, \ldots, X_{n-1} X_n ).
\]
Now consider $\widetilde I =F^\perp : I + (t)$, where  $t=\alpha_1 X_1+ \ldots+\alpha_n X_n \in I_1$  is a general form.
 We have
\[\widetilde I = F^\perp : I + (\alpha_1 X_1+ \ldots+\alpha_n X_n) \]
\[=(X_0 ^{a+1},X_1^2,\ldots,X_n^2, X_1X_2, \ldots, X_{n-1} X_n , \alpha_1 X_1+ \ldots+\alpha_n X_n) .
\]
We want to apply Corollary \ref {maincor}, so we
 compute  $\sum _{i=0} ^s HF (T/  \widetilde I , i)$ for $s$ large enough.

 For $a+1 = 2$ and $b=2$,  $F=  x_0 (x_1^2 + \ldots + x_n^2)$ and
 \[ \widetilde I=
 (X_0 ^{2},X_1^2,\ldots,X_n^2, X_1X_2,   X_1X_3 , \ldots, X_{n-1} X_n , \alpha_1 X_1+ \ldots+\alpha_n X_n) .
 \]
 So we can easily see that
  \[
 \begin {array}{c|ccccc}
 i                                    &       0 & 1 & 2 &3  \\
 \hline \\
 HF (T/\widetilde I , i)  & 1  & n  & n-1 & 0 \\
 \end {array}
 \]
From this we get $\sum _{i=0} ^s HF (T/\widetilde I  , i)= 2n$.
\vskip 2pt

 For $a+1 > 2$ we have
 \[ \widetilde I=(X_0 ^{a+1},X_1^2,\ldots,X_n^2, X_1X_2, X_1X_3\ldots, X_{n-1} X_n , \alpha_1 X_1+ \ldots+\alpha_n X_n) .
\]
A simple computation shows that
 \[
 \begin {array}{c|cccccc}

i                                    &  0 & 1 & 2 \ldots & a & a+1 &a+2 \\
\hline \\
 HF (T/\widetilde I , i) & 1  & n & n  \ldots & n & n-1 &0 \\
 \end {array}
 \]
From this we get $\sum _{i=0} ^{s} HF (T/\widetilde I  , i)= (a+1)n$.

Hence, we get $\mathrm{rk} (F) \geq  (a+1)n$ in both cases using Corollary \ref {maincor},.

Now consider $F^\perp$. Since
\[F^\perp \supseteq (X_0 ^{a+1}, X_1^b-X_2^b,\ldots, X_1^b-X_n^b, X_1X_2, X_1X_3,  \ldots, X_{n-1} X_n ) ,
\]
then the  ideal
\[
(X_0 ^{a+1}+  (X_1^{a+1-b} +\ldots+X_n^{a+1-b} )
((n-1) X_1^b-\ldots -X_n^b), X_1X_2, X_1X_3, \ldots, X_{n-1} X_n )
\]
is contained in $F^\perp$. This last  is the ideal of $(a+1)n$ distinct points lying on the $n$ lines whose defining ideal is $(X_1X_2, X_1X_3, \ldots, X_{n-1} X_n ) $.

 By the Apolarity Lemma, it follows that $\mathrm{rk}(F) \leq (a+1)n$, and we are done.

\end{proof}

\begin{rem}

 For some special $F$ in Proposition \ref{a+1geqb} the rank of $F$ can be computed by $t$, instead of by $I$ and $t$. For instance, if $F= x(y^2+z^2)$ we have $\rk(F)=4$. Note that in the proof of Proposition \ref {a+1geqb} we showed that the rank was computed by  $I=(Y,Z)$ and $t= \alpha_1 Y+ \alpha_2 Z$. However,  the rank is also computed by $t=X$, in other words:
$$\sum_{i=0}^\infty HF(T/(F^\perp: (X) +(X)),i) =4.$$

We do not know if the rank of $F$ can always be computed by $t$. For instance, if $F= x^2(y^2+z^2+w^2)$ we have $\rk(F)=9$. In the proof of Proposition \ref {a+1geqb} we showed that the rank was computed  by $I=(Y,Z,W)$ and $t= \alpha_1 Y+ \alpha_2 Z+ \alpha _3 W$. Note that
$$\sum_{i=0}^\infty HF(T/(F^\perp: (Y+Z+W) +(Y+Z+W)),i) =3,$$

and that

$$\sum_{i=0}^\infty HF(T/(F^\perp: (X) +(X)),i) =5,$$
that is, neither $t=Y+Z+W$, nor $t=X$ compute the rank.
We do not know if there is a $t$ which computes the rank of this $F.$

\end{rem}

%%%%%%%%%%%%%%%%%%%%%%%%%%%%
%%%%%%%%%%%%%%%%%%%%%%%%%%%%%%

\begin{rem} Let $M_i = x_0^ax_i^b$, so the polynomial $F$ of the previous proposition, becomes
\[
F=x_0^a(x_1^b + \ldots + x_n^b) = M_1+ \ldots +M_n.
\]
In case $a+1 = b$ we have (see \cite{carcatgermonomi} for the rank of the $M_i$)
$$
\rk(F) = (a+1)n < \rk(M_1) + \cdots + \rk(M_n)= (a+2)n.
$$
Thus,   an analogue of Strassen's Conjecture is certainly not true if a form is the sum of forms which have a common factor.  On the other hand,
when   $a+1 >b$, we have
\[ (a+1)n  =\mathrm{rk} (F) \leq \mathrm{rk}(M_1)+ \ldots +\mathrm{rk} (M_n )= (a+1)n.
\]
Thus, in some cases, the rank is additive over summands, even when the summands have a common factor.
\end{rem}

\begin {prop}\label{n=2}

Let $ b \geq 2$, $a\geq 1$, and let
\[F =x_0^a(x_1^b + x_2^b).\]

 $(i)$ If $a+1 \geq b$, then the rank of $F$ is computed by $I=(X_1,X_2)$ and $t$ and
 $ \mathrm{rk} (F) = 2(a+1).$

$(ii)$ If $a+1 \leq b$, then the rank of $F$ is computed by $t=X_0$  and
 $\mathrm{rk} (F) = 2b.$

\end{prop}

\begin {proof}
$(i)$ Follows from Proposition \ref{a+1geqb}.

$(ii)$ In this case let $I= (X_0) \subset T $.
Obviously  $t $   is a general form in $I_1$.
Hence we consider the ideal
  $\widetilde I = F^\perp :  (X_0) + (X_0)$, and we have
\[\widetilde I = (X_0 \circ F)^\perp + (X_0)=(x_0 ^{a-1} (x_1^b+x_2^b))^\perp+ (X_0)  \]
\[= (X_0, X_1X_2, X_1^b-X_2^b ). \]

Since
 \[
 \begin {array}{c|cccccc}

i & 0 & 1 & 2& \ldots &b-1& b  \\
\hline \\
 HF (T/\widetilde I  , i)  & 1 & 2 & 2& \ldots & 2& 1 \\
 \end {array}
 \]
we have $\sum _{i=0} ^b HF (T/\widetilde I  , i) =2b$.
Hence from Corollary \ref {maincor},  we get $\mathrm{rk} (F) \geq 2b$.

Since
\[  (X_1X_2,X_0^b + X_1^b - X_2 ^b )
\]
is the ideal of  $2b$ points  apolar to $F$, by the Apolarity Lemma we are done.

\end {proof}

\begin{rem}

Note that for $a+1 \leq b$ and $F =x_0^a(x_1^b + x_2^b)$
\[\mathrm{rk} (F) = 2b <  \mathrm{rk}(x_0 ^a x_1^b )+\mathrm{rk}(x_0 ^a x_2^b )=2b+2.\]

\end{rem}

Now we study the rank of the forms $G=F+x_0^{a+b}$,  where $F$ is as in  Propositions \ref{a+1geqb} and \ref{n=2}, that is,
$$
G = x_0^a(x_1^b + \ldots + x_n^b)+ x_0^{a+b} .
$$
We will show that $F$ and $G$ have the same rank.

\begin {prop} \label{a+1geqb2}

Let $ b \geq2$, $n \geq 2$ and let
\[G = x_0^a(x_1^b + \ldots + x_n^b)+x_0^{a+b}
=x_0^a(x_0^b+x_1^b + \ldots + x_n^b) \in S.\]
If $a+1 \geq b$, then the rank of $G$ is computed by $I=(X_1, \ldots , X_n)$ and $t$ and
 \[ \mathrm{rk} ( G )= (a+1)n.\]

\end{prop}

\begin {proof}
As in the proof of  Proposition \ref{a+1geqb}, we consider the ideal $I= (X_1, \ldots, X_n) \subset T$ and the linear general form $t =\alpha_1 X_1+ \ldots+\alpha_n X_n$. Let  $\widetilde I = G^\perp:I+ (t)$. We have
\[ \widetilde I =G^\perp :  (X_1, \ldots, X_n) + (\alpha_1 X_1+ \ldots+\alpha_n X_n)\]
\[= (G^\perp :  (X_1) )\cap \cdots \cap( G^\perp :  (X_n)) + (\alpha_1 X_1+ \ldots+\alpha_n X_n).\]
Hence, by Lemma \ref {colonperp},
\[
\widetilde I = (x_0^a x_1^{b-1}) ^\perp \cap \cdots \cap(x_0^a x_n^{b-1}) ^\perp + (\alpha_1 X_1+ \ldots+\alpha_n X_n).
\]
Note that this is exactly the ideal $\widetilde I$ that we constructed in the proof of Proposition \ref{a+1geqb}, thus we may proceed in the same way and we get
 $\mathrm{rk} (G) \geq  (a+1)n$.

Now consider $G^\perp$. It is easy to show that $G^\perp$ contains the ideal
\[
( nX_0 ^{a+1}- {a+b \choose b}( X_1^b+ \ldots +X_n^b) X_0^{a+1-b} , X_1^{b+1},  \ldots , X_n^{b+1},
\]
\[
 X_1X_2, X_1X_3,\ldots, X_{n-1} X_n).
\]

If $a+1=b$, then the ideal
\[    (nX_0 ^{a+1}- {a+b \choose b }( X_1^{b}+ \ldots +X_n^b) X_0^{a+1-b}, X_1X_2, X_1X_3,\ldots, X_{n-1} X_n ) ,\]
 is contained in $G^\perp$ and defines  $(a+1)n$  points apolar to $G$  lying on the $n$ lines whose defining ideal is $(X_1X_2, X_1X_3, \ldots, X_{n-1} X_n ) $. Hence, we conclude using the Apolarity Lemma.

If $a+1>b$, then consider the   ideal
\[
\mathcal{ A } = (\alpha(nX_0 ^{a+1}-  {a+b \choose b }X_0^{a+1-b}( X_1^{b}+ \ldots +X_n^b)) + \beta X_1^{a+1}+ \ldots +\beta X_n^{a+1},
\]
\[ X_1X_2, X_1X_3,\ldots, X_{n-1} X_n ) ,\]
where $\alpha, \beta \in k$. It is easy to see that $\mathcal A$ is contained in
$  G^\perp$. Moreover, for  generic values of $\alpha$ and $\beta$, $\mathcal A$ is the ideal of $(a+1)n$ distinct points lying on the $n$ lines whose defining ideal is $(X_1X_2, X_1X_3,\ldots, X_{n-1} X_n ) $. In fact, consider the line whose ideal is $ (X_2, \ldots,  \ X_n)$ (and analogously for the other $n-1$ lines). We have
\[\mathcal{A} + (X_2, \ldots,X_n) =
(\alpha (nX_0 ^{a+1}-  {a+b \choose b }X_0^{a+1-b}X_1^{b} )+ \beta  X_1^{a+1}, X_2, \ldots,  X_n ) ,
\]
hence, in order to find the $a+1$ points, we have to solve the equation
\[\alpha (nX_0 ^{a+1}-  {a+b \choose b }X_0^{a+1-b}X_1^{b}) + \beta X_1^{a+1}=0
,\]
or, in other words, we have to consider the linear series  cut out on $\PP^1$ by the linear system
\[
\Sigma =<nX_0 ^{a+1}-  {a+b \choose b }X_0^{a+1-b}X_1^{b} , \  X_1^{a+1}>
,\]
whose general element is reduced by Bertini's Theorem.

Thus, using the Apolarity Lemma, it follows that $\mathrm{rk}(G) \geq (a+1)n$, and we are done.
\end{proof}

\begin{rem}

 The lower bound in \cite {LandsbergTeitler2010}, [Proposition 4.7] can only prove the case a=1 and b=2 of our Proposition 4.9.

\end{rem}

\begin {prop}\label{n=2parte2} Let $ b \geq2$ and
\[G =x_0^a(x_1^b + x_2^b)+x_0^{a+b}=x_0^a(x_0^b+x_1^b + x_2^b)\in S.\]

 $(i)$ If $a+1 \geq b$, then the rank of $G$ is computed by $I=(X_1,X_2)$ and
a general $t\in I_1$,  and
 $ \mathrm{rk} (G) = 2(a+1).$

$(ii)$ If $a+1 \leq b$, then the rank of $G$ is computed by $t=(X_0)$ and
 $\mathrm{rk} ( G) = 2b.$

\end{prop}

\begin {proof}

$(i)$ This is a particular case of Proposition \ref{a+1geqb2}.

$(ii)$ As in Proposition \ref{n=2}, let $I = (X_0) $ and $t=X_0$.
Consider the ideal $\widetilde I = G^\perp :  (X_0) + (X_0)$. We have

\[\widetilde I = (X_0 \circ G)^\perp + (X_0)=(x_0 ^{a-1} (x_1^b+x_2^b))^\perp+ (X_0)  \]
\[= (X_0, X_1X_2, X_1^b-X_2^b ) ,\]
which is the same ideal we found in the proof of Proposition \ref{n=2}. So $\mathrm{rk} (G) \geq 2b$ follows in the same way.

Now notice that
\[ (2X_0 ^{b}- {a+b \choose b}( X_1^b+X_2^b), X_1X_2 )
\]
is the ideal of $2b$  points which are apolar to $G$. Thus, by the Apolarity Lemma, $\rk(G) \leq 2b$, and we are done.

\end {proof}

\begin {rem}

With a bit more effort one can show the following:

a) In Propositions \ref {a+1geqb}, \ref{n=2} (i), \ref{a+1geqb2} and \ref {n=2parte2} (i)
the forms are $e$-computable if $2e \leq b$. The rank  is computed by $I=(X_1^e,\ldots, X_n^e)$ and a general form $t \in I_e$.

b)  In Propositions  \ref{n=2} (ii) and \ref {n=2parte2} (ii)
the forms are $e$-computable if
$2e \leq a+1$ and the rank of $F$ is computed by $I = (X_0^e)  $ and $t=X_0^e$.

\end {rem}

Now we  study  forms $F \in S=k[x_0 , \ldots  ,x_n]$  for which
\[F^\perp=(q^a,g_1,\ldots,g_n) \subset T\]
is a complete intersection such that
\[a\geq 2  \ \hbox {and }\  ae \leq d_1\leq\ldots \leq d_n,\]
where $ e=\deg q , \  d_1=\deg g_1, \ldots ,  d_n=\deg g_n.$

We need the following lemma:

\begin{lem}\label{cilem}
Let $J =(q^a,g_1,\ldots,g_n)$ be a complete intersection  as above.  Then there exist
$f_1,\ldots , f_n$ such that
$$J = (q^a, f_1,\ldots , f_n),$$
 where $\deg f_i = \deg g_i$ and, for all $j, \ 1 \leq j \leq n$ the ideal $(f_j, f_{j+1}, \ldots , f_n)$ defines a smooth complete intersection in ${\mathbb P}^n$ of codimension $n-j+1$ and having degree $\Pi_{i=j} ^d d_i$.

\end{lem}

\begin{proof}

Consider the linear system of forms of degree $d_n$ in $J$.  This system has no base points and so by Bertini's Theorem, the general element is smooth.  Since the general element is a linear combination of $g_n$ and other forms of degree $d_n$ in $J$, there is no loss of generality in choosing a generator for $J$ of the type $f_n=g_n + \hbox{(other forms of degree } d_n)$.  We call this new generator $f_n$.  Now consider the linear system of codimension two varieties cut out on $V(f_n)$ by all the other hypersurfaces in $J$ of degree $d_{n-1}$.  This linear system is clearly base point free in $V(f_n)$ and so the general element of this system cuts out a smooth variety on $V(f_n)$ of codimension 2 in ${\mathbb P}^n$.  We can then replace $g_{n-1}$ by a general element of this system.  Continuing in this same way we arrive at hypersurfaces $f_1, \ldots , f_n$ where $\deg f_i = \deg g_i$ and $(f_1, \ldots , f_n)$ describes a set of $\Pi_{i=1}^n d_i$ points.

\end{proof}

We now want to study sets of points apolar to $F$, having  some points which lie on the variety defined by $q=0$.  We have the following result.

\begin{thm} \label{teor-ci}
Let $F\in S$ be a homogeneous polynomial. If
\[F^\perp=(q^a,g_1,\ldots,g_n)\]

 is a complete intersection such that
\[ a\geq 2 \ , \ \ e=\deg q>0 \hbox{ and } ae \leq d_1=\deg g_1\leq\ldots \leq d_n=\deg g_n,\]

then  $F$ is $e$-computable, the rank of $F$ is computed by $q$  and we have
\[\rk (F)=\Pi_1^n d_i={(1/e) } \sum_{i=0}^\infty HF(T/(F^\perp:(q) + (q)),i ).\]

\end{thm}

\begin{proof}
Using Lemma \ref{cilem} we know that $\rk (F)\leq \Pi_1^n d_i$.

Since $\{q^a,g_1,\ldots,g_n\}$ are a regular sequence,
$F^\perp : (q) = (q^{a-1},g_1,\ldots,g_n)$. Hence
\[ F^\perp:(q)+(q)=(q,g_1,\ldots,g_n).\]

So by Corollary \ref{maincor}  we have
\[\rk (F) \geq
\left ({1\over e} \right)  \sum_{i=0}^\infty HF(T/ (q,g_1,\ldots,g_n), i )= {\Pi_1^n d_i},\]
and the conclusion follows.

\end{proof}

 We now give an example of a form which is $2$-computable but not $1$-computable.

 \begin{ex} \label{example2}
If $F=$
\[
\begin{array}{c}
x^{11} -22x^{9}y^{2} + 33x^{7}y^{4} -22x^{9}z^{2} + 396x^{7}y^{2}z^{2}
-462x^{5}y^{4}z^{2} + \\ 33x^{7}z^{4} -462x^{5}y^{2}z^{4} + 385x^{3}y^{4}z^{4},
\end{array}
\]

then $F$ is $2$-computable and $\rk(F)=25$.
In fact, using \cite{cocoa}, we get  $$F^\perp = ((X^2 +Y^2+Z^2)^2, Y^5, Z^5),$$ hence
$$
 \rk(F)\geq  (1/2) \sum_{i=0}^\infty HF(T/ (F^\perp  :(X^2 +Y^2+Z^2) +(X^2 +Y^2+Z^2)),i)= 25,
 $$
and the ideal $(Y^5 +Z(X^2 +Y^2+Z^2)^2,  Z^5 + X(X^2 +Y^2+Z^2)^2) \subset F^\perp$ is the ideal of $25$ distinct points.

We will see, in Example \ref{example1}, that this form is not 1-computable.

 \end{ex}

\begin{prop}\label{CI-e-una-variab}

Let $F= x_0^a G \in S$  for some $a$ and some form $G \in k[x_1,...,x_n]$. The following hold:

i) $F^\perp = (X_0^{a+1} , G^\perp)$, where $G^\perp$ is considered in $k[X_1,...,X_n]$.

ii) If $G^\perp$ is a complete intersection and all generators of  $G^\perp$ have degree at least $a+1$, then $F$ is $1$-computable.
\end{prop}

\begin{proof} First of all, let $g \in F^\perp$. We can write $g = h_0 + X_0h_1+\cdots+X_0^ah_a+ X_0^{a+1} \widetilde g$ where $ h_0,..,h_a \in k[X_1,...,X_n]$ and $\widetilde g \in k[X_0,...X_n]$. By assumption,
\begin{align*}
  0 &= g\cdot F \\
      &= (h_0 + X_0h_1+\cdots+X_0^ah_a+ X_0^{a+1}\widetilde g)\cdot x_0^aG(x_1,...,x_n)\\
      &= x_0^a(h_0 \cdot G) + ax_0^{a-1}(h_1\cdot G)+\cdots+ (a!)(h_a\cdot G). \\
\end{align*}
Since $h_0 \cdot G, h_1\cdot G, ...,h_a \cdot G \in \mathbb{C}[x_1,...,x_n]$, we have $h_0 \cdot G= h_1\cdot G= ...=h_a \cdot G=0$ and hence $h_0,...,h_a \in G^\perp $. This proves that $F^\perp = (X_0^{a+1}, G^\perp)$.

ii) Obvious from Theorem \ref{teor-ci}.
\end{proof}

Let $V_n = \prod_{1\leq i < j \leq n} (x_i -x_j) \in k[x_1,...,x_n]$ be the Vandemonde determinant. Since $V_n$ is the fundamental skew-symmetric invariant of the symmetric group, it is known that the perp ideal $V_n^{\perp} = (\sigma_1,\sigma_2,..,\sigma_n) \subset k[X_1,...,X_n]$ where $\sigma_i$ is the $i$-th elementary symmetric polynomial in $X_1,...,X_n$  for $i=1,...,n$ (see \cite {TW} and its bibliography). For  later use, let $\sigma'_i$ be the $i$-th elementary symmetric polynomial on the variables $X_2,...,X_n$ for $i = 1,...,(n-1)$. One can see that
\begin {itemize}
 \item $\sigma_1 = X_1 + \sigma'_1$,
 \item $\sigma_2 = X_1 \sigma'_1 + \sigma'_2$,
 \item $\cdots$
 \item $\sigma_{n-1} = X_1 \sigma'_{n-2} + \sigma'_{n-1}$,
 \item $\sigma_n = X_1\sigma'_{n-1}$.
 \end  {itemize}

\begin{prop}\label{Teitler-Woo} \cite{TW} \ $\rk(V_n) = (n-1)!$
\end{prop}

\begin{proof} We will give an elementary proof,  different from that in
\cite {TW}, which uses  the Apolarity lemma.
We have $\rk(V_n) \geq (n-1)!$ by the Ranested-Schreyer bound (see \cite{RS00}). For the upper bound, take $I = (\sigma_1,...,\sigma_{n-1}) \subset V_n^{\perp} $. By the Apolarity lemma, it remains to show that $I$ is the homogenous ideal of a set of $(n-1)!$ distinct points. To this end, we will show that on the affine piece $X_1 \neq 0$, the zero locus of the ideal $I$ consists of exactly $(n-1)!$ distinct points. This is enough because $I$ is a complete intersection of forms of degrees $1,2,...,(n-1)$. Now letting $X_1=1$, we have
 $$ \{(X_2,...,X_n) | \sigma_1(1,X_2,...,X_n)=\cdots= \sigma_{n-1}(1,X_2,...,X_n)=0 \}$$
$$  =\{(X_2,...,X_n) | 1+\sigma'_1(X_2,...,X_n)=\cdots= \sigma'_{n-2}(X_2,...,X_n)+\sigma'_{n-1}=0 \} $$
 $$=\{(X_2,...,X_n) | \sigma'_1=-1,...,\sigma'_i=(-1)^i,...,\sigma'_{n-1}=(-1)^{n-1} \}$$
 $$=\{(X_2,...,X_n)| X_2,...,X_n \text{ are the distinct } (n-1) \text{ solutions of the equation }$$
 $$ \mathbf{t}^{n-1}+\cdots +\mathbf{t}+1=0 \}.$$
 This proves that the ideal $I$ defines a set of $(n-1)!$ distinct points.
\end{proof}

 \begin{prop}The rank of the Vandemonde determinant $V_n$ is computed by the linear form $X_1$.
 \end{prop}

\begin {proof} In light of Proposition \ref {Teitler-Woo} it will be enough to show that the length of
$T/(V_n^{\perp}:(X_1)+(X_1)) = (n-1)!$. We first observe that since
$\sigma_1,...,\sigma_{n}$ form a regular sequence and $\sigma_n = X_1\sigma'_{n}$ we have that both $\sigma_1,...,\sigma_{n-1}, X_1 $ and $\sigma_1,...,\sigma_{n-1}, \sigma'_n$ form regular sequences. It is also clear that
$$V_n^{\perp}+(X_1) = (X_1,\sigma_1,...,\sigma_{n-1}, \sigma_n) =(X_1,\sigma'_1,...,\sigma'_{n-1} ). $$
Obviously $X_1,\sigma'_1,...,\sigma'_{n-1}$ is a regular sequence and so
$$ \sum_{i=0}^\infty HF (T/ (V_n^{\perp}+(X_1)),i) = (n-1)! .$$
Thus from the exact sequence
$$ 0 \rightarrow T/(V_n^{\perp}:(X_1)) \rightarrow  T/V_n^{\perp} \rightarrow  T/(V_n^{\perp}+(X_1)) \rightarrow 0$$
we obtain
$$\sum_{i=0}^\infty HF (T/(V_n^{\perp}:(X_1)),i) = n!-(n-1)! = (n-1)!\cdot(n-1).$$
Now notice that
$$V_n^{\perp}:(X_1) \supseteq (\sigma_1,...,\sigma_{n-1}, \sigma'_{n-1}).$$
But the {length}  of $T/(V_n^{\perp}:(X_1)) $ is $ (n-1)(n-1)!$ and this is exactly the length of
$T/ (\sigma_1,...,\sigma_{n-1}, \sigma'_{n-1})$. It follows that
$$V_n^{\perp}:(X_1) = (\sigma_1,...,\sigma_{n-1}, \sigma'_{n-1}).$$
Hence $V_n^{\perp}:(X_1)+(X_1) = (X_1, \sigma_1,...,\sigma_{n-1}, \sigma'_{n-1})$ and this is easily seen to be $V_n^{\perp}+(X_1)$. But we have already shown that
$\sum_{i=0}^\infty HF(T/(V_n^{\perp}+(X_1)),i)= (n-1)! $ and thus we are done.

\end {proof}

Note that the Vandemonde determinant is  1-computable and in $V_n^\perp$ there is a form of degree one.
A  natural question arises:
Does there exist a change of coordinates  such that, after this change, we may consider  $V_n$ in a smaller polynomial ring, in which $V_n$ is still $1$-computable and $(V_n^\perp)_1 =0$?

In  Proposition \ref{Fperp1} we  give a positive answer to this question, but first we observe the following:

\begin {rem}
Recall that $T=k[X_0,...,X_n]$ and suppose that $Y_0,\ldots, Y_n$ is another basis for $T_1$, where
$$ Y_i =\sum_{i=0}^n \alpha _{i,j} X_j.$$
We can write $T$ as a polynomial ring in the new variables $Y_0,...,Y_n$. To avoid confusion we set $\widetilde T = k[Y_0,...,Y_n]$, even though $\widetilde T =T$.
The change of coordinates transformation on $T$ can be considered as
$$\psi : T \rightarrow \widetilde T$$
where
$$X_i  = \psi_i (Y_0,...,Y_n) .$$
It follows that, for a form $G(X_0,...,X_n) \in T$,
$$\psi (G) = G(\psi_0( Y_0, \dots, Y_n),\dots, \psi_n( Y_0, \dots, Y_n))  \in \widetilde T
.$$
Now let $y_0,\ldots, y_n \in S_1$ be the dual basis to $Y_0,\ldots, Y_n$. As with the discussion above we can consider
$$\varphi : S = k[x_0,...,x_n] \rightarrow  \widetilde S = k[y_0,...,y_n] $$
the isomorphism which extends the isomorphism induced by $\psi$ from
$S_1 \rightarrow  \widetilde S_1$.

Since $X_i \circ x_j = \delta_{i,j}$ and $Y_i \circ y_j = \delta_{i,j}$, we have, for $G\in T$ and $F \in S$,

\[ \varphi (G \circ F) = \psi (G) \circ \varphi (F).
\]
\end {rem}

\begin {lem} \label{dualbasis}
Let   $Y_0, \ldots ,Y_n$ be a basis  for $T_1$ and let  $y_0, \ldots ,y_n \in S_1$ be the dual basis.
Let
$\widetilde T = k[Y_0,\ldots, Y_n]$, $\widetilde S = k[y_0,\ldots, y_n]$, and let
 $\psi:  T \rightarrow \widetilde T $ and $\varphi:  S \rightarrow \widetilde S $
be the changes of coordinates.

If  $F(x_0,\dots, x_n)  \in S$ then
\[ \psi( F^\perp ) =  \varphi (F) ^\perp .
\]
\end{lem}
\begin {proof}
Let  $F^\perp = (G_1,\ldots, G_s)$, so $\psi( F^\perp ) =  ( \psi (G_1), \dots ,  \psi (G_s) )$. Since
 $ \psi (G_i) \circ \varphi (F) = \varphi (G_i \circ F)=0$,  we get  $\psi (G_i) \in \varphi (F)^\perp$.
For the opposite inclusion,  let $\widetilde G  \in \varphi (F)^\perp$, and
 $G =\psi ^{-1} ( \widetilde G) $.   We have that
$ \psi (G) \circ \varphi (F) = 0$. But $ \psi (G) \circ \varphi (F)= \varphi (G \circ F)$, hence $G \circ F=0$, that is, $G \in F^\perp$, and so
$\widetilde G \in \psi( F^\perp )$.
\end{proof}

\begin{prop} \label{Fperp1}

Let  $F\in S= k[x_0,\dots,x_n]$  and assume that
$$ (F^\perp)_1=(Y_{n-s+1}, \dots, Y_n) \subset T_1,$$ where the $Y_i$ are linearly independent linear forms in the $X_i$.

Let $Y_0, \ldots ,Y_{n-s}, Y_{n-s+1}, \dots, Y_n$ be a basis of $T_1$
and let
$y_0, \ldots ,y_n$ $ \in S_1$ be its dual basis.
There exists  a change of coordinates $\varphi$ such that $\varphi(F) $ involves only the variables $y_0,\dots,y_{n-s}$, and  considering $\varphi(F)$ in $k[y_0,\dots,y_{n-s}]$, we have $(\varphi(F)^\perp)_1=0$.
Moreover, if $F$ is $1$-computable, then $\varphi(F) $ also is $1$-computable.

\end{prop}

\begin {proof}

Let $\varphi$ and $\psi$ be as in Lemma \ref {dualbasis}, then
we get
\[ (\psi( F^\perp ) )_1= ( \varphi (F)^\perp)_1 .
\]
Since  $(\psi( F^\perp ) )_1 =(Y_{n-s+1}, \dots, Y_n) \subset \widetilde T_1 $, we have that $Y_i \circ \varphi (F) =0$ for  $n-s+1 \leq i \leq n$. It follows that
$\varphi (F)  \in k[y_0,\dots,y_{n-s}]$.
Now assume that $F$ is $1$-computable, and that the rank of $F$ is computed by $I$ and $t$,
that is, $$\rk (F)= \sum_{i=0}^\infty HF (T/ (F^\perp : I + (t)) .$$
 Since
$\psi (F^\perp : I + (t)) =\psi (F^\perp)  :\psi ( I )+ \psi (t) = \varphi (F)^\perp  :\psi ( I )+ \psi (t)$ and
$T/ (F^\perp : I + (t)) \simeq \widetilde T/ (\psi (F^\perp : I + (t)))=  \rk \varphi (F) $, then
$\varphi (F)$ is $1$-computable, and we are done.

\end{proof}

\begin {rem} \label{remVand}

By a change  of coordinates $\varphi$  as in Proposition  \ref {Fperp1},
we may assume that the form $\varphi (V_n)$, where $V_n$ is the Vandermonde determinant, is 1-computable and $(\varphi (V_n)^\perp)_1=0. $

\end {rem}

We close this section by exhibiting a family of forms which are $e$-computable ($e>1$) but are not  1-computable.

\begin{ex} \label{example1}
Let $T$ be a polynomial ring in three variables. Let $Q\in T$ be
 an irreducible quadratic   form and
let $G_1,G_2\in T$ be two general forms of degree $d$, $d >4 $.
By Macaulay duality, there exists a form $F$ in the dual ring $S$
whose apolar ideal is $$F^\perp= (Q^2,G_1,G_2).$$
By Theorem \ref{teor-ci} we know that $F$ is $2$-computable and $\rk (F)=d^2$.

We claim that $F$ is not 1-computable.

Note that $(G_1,G_2)\subset F^\perp$ is the ideal of a set of $d^2$ distinct points, say $\mathbb{X}$. By Proposition \ref{prop1ecomp}, if $F$ were $1$-computable by $I$ and $t$ ($t$ general in $I$), then
$$ I_\Bbb X + (t) = F^\perp +(t).$$
Thus, we would have then $(G_1,G_2,t) = (Q^2,G_1,G_2,t)$, which is impossible since $t$ does not divide $Q$. Hence $F$ is not $1$-computable.

\end{ex}

%%%%%%%%%%%%%%%%%%%%%%%%%%%%%

%%%%%%%%%%%%%%%%%%%%%%%%%%%%%%%%%%

\begin{ex}  \label{remNonLinComp} In  Section \ref{examples} we  exhibit a form  $F $ whose rank we can compute using ad hoc methods. We show it is not  $1$-computable and wonder if it is  $e$-computable for some $e>1$.
\end{ex}

%%%%%%%%%%%%%%%%%%%%%%%%%%%%%%%%%%%%%%%%%%%%%%%
%%%%%%%%%%%%%%%%%%%%%%%%%%%%%%%%%%%%%%%
%%%%%%%%%%%%%%%%%%%%%%%%%%%%%%%%%%
%%%%%%%%%%%%%%%%%%%%%%%%%%%%%%%%%%

\section{Strassen's conjecture for $\MakeLowercase{e}$-computable forms} \label {mainresultsececomp}

%%%%%%%%%%%%%%%%%%%%%%%%%%%%%%%%%%

Fix the following notation:
\[S=k[x_{1,0} ,\ldots, x_{1,n_1},
\ldots \ldots
, x_{m,0} ,  ,\ldots, x_{m,n_m} ],\]
\[T=k[X_{1,0} ,\ldots, X_{1,n_1},
\ldots \ldots
,X_{m,0}   ,\ldots, X_{m,n_m} ].\]

For $i = 1, \dots , m$, we let
\[ S^{[i]} = k[x_{i,0} ,\ldots, x_{i,n_i}]  ,\]
\[ T^{[i]} =k[X_{i,0} ,\ldots, X_{i,n_i} ],\]
\[F_i \in  S^{[i]}_d, \]
and
\[ F = F_1 + \cdots + F_m \in S_d .\]

If we consider  $F_i \in S$, then we write
\[F_i^\perp=\left\{g\in T \mid g\circ F_i=0\right\}.\]
On the other hand, if we consider $F_i \in S^{[i]}$, then we  also write
\[F_i^\perp=\left\{g\in T^{[i]} \mid g\circ F_i=0\right\}.\]

Given this notation, it is important to know precisely in which ring we are considering $F_i$.

So, for instance, if we consider $F_1 \in S$ then
\[F_1^\perp=\left\{g\in T^{[1]} \mid g \circ F_1=0\right\} \cup
(X_{2,0} ,\ldots, X_{2,n_2},
\ldots \ldots
,X_{m,0}   ,\ldots, X_{m,n_m} ) ,\]
while if we consider $F_1 \in S^{[1]}$ then
\[F_1^\perp=\left\{g\in T^{[1]} \mid g \circ F_1=0\right\}. \]

\begin {rem}\label{rem2}
We assume  that each $F_i$ essentially involves $n_i$  variables, thus  $F_i^\perp$  does not have linear forms involving the variables of $ T^{[i]}$, and in  $F^\perp$  there are no linear forms.
\end {rem}

Moreover, we let $I^{[i]} \subset T^{[i]} $ be  ideals with $t_i \in I^{[i]}$ ($i=1,\cdots,m$) all of the same degree and we set
\[J_i = (F_i^\perp :I^{[i]})+(t_i)  \subset T.
\] where we consider each $F_i$ as a form in $S$.

\begin {lem} \label {lemma1ecomp}
With the notation above  and $a_i \in k$ we have
 \[(F^\perp : (I^{[1]}+ \dots+ I^{[m]}) )+(a_1t_1+ \cdots+  a_mt_m) \subseteq
 J_1 \cap \dots \cap J_m .\]
\end {lem}
\begin {proof}
Since $F_i \in S^{[i]}$ (although we are considering it in $S$) we always have that $X_{j,0}, \ldots , X_{j,n_j}$ are in $F_i^\perp$ for all $j \neq i$.  Hence $t_j \in F_i^\perp$ for $j\neq i$.  So
 $t_1,\ldots,t_m \in J_1 \cap \dots \cap J_m$ and  it is enough to prove that
\[(F^\perp : (I^{[1]}+ \dots+ I^{[m]}) ) \subseteq
 J_1 \cap \dots \cap J_m,\]
 that is,
 \[(F^\perp : I^{[1]})  \cap \dots \cap (F^\perp :  I^{[m]}) \subseteq
 J_1 \cap \dots \cap J_m.\]
 Let $g \in  F^\perp : I^{[i]}$, ( $1 \leq i \leq m$),
 so $gl \circ F=0$, for any $l \in I^{[i]}$. Since for $j \neq i$,
$l \circ F_j=0$, then
 $gl \circ F_i=0$, that is,  $gl \in F_i^\perp$, by considering $F_i \in S.$ It follows that
 $g \in F_i^\perp :I^{[i]} \subseteq  J_i$,  for $i=1,\ldots,m$, that is,
 $g \in  J_1 \cap \dots \cap J_m $.
\end {proof}

\begin {lem} \label {lemma2ecomp}
Let $J_i$ be as above. If $s \gg 0$, then
 \noindent
(i)
\[\sum_{i=0}^s HF(T/  J_1\cap \ldots \cap J_m ,i)=  \sum_{i=0}^s HF(T/  J_1 ,i)+ \ldots + \sum_{i=0}^s HF(T/  J_m ,i)-m+1
.
\]
(ii)
If $t_i \in I^{[i]}$ is a general form and the rank of $F_i$ is computed by $I^{[i]}$ and $t_i$, then
\[\sum_{i=0}^s HF(T/  J_1\cap \ldots \cap J_m ,i)=
e(\mathrm{rk} (F_1)+ \cdots + \mathrm{rk} (F_m))-m+1
;
\]
\end {lem}
\begin{proof}
To prove (i) we proceed by induction on $m$. If $m=1$ the equality is obvious. Let $m>1$ and consider
 the following short exact sequence:
 \[
0\longrightarrow T/ (J_1\cap \ldots \cap  J_m)  \longrightarrow
T/J_1\oplus T/ (J_2\cap \ldots \cap  J_m)
\longrightarrow T/(J_1+ (J_2\cap \ldots \cap  J_m) )\longrightarrow 0.
\]
Since $J_1+ J_2\cap \ldots \cap J_m$ is the maximal ideal of $T$ we get the conclusion by the inductive hypothesis.

(ii) follows from (i) since $T/J_i \simeq T^{(i)}/ F_i ^\perp : I^{[i]}+(t_i)$, where  now $F_i$ is considered as a form in $S^{[i]}$
(so $F_i^\perp=\left\{g\in T^{[i]}  \mid g \circ F_i=0\right\}$). Hence, for $s\gg 0$, we have
 \[ e\cdot\mathrm{rk} (F_i) =
  \sum_{j=0}^s HF(T^{[i]}/F_i^\perp: I^{[i]}+(t_i),j) .\]
\end{proof}

\begin {rem} \label {rem3}
Recall that in \cite[Proposition 3.1]{CCC14},  it was shown that Strassen's conjecture holds for foms of the type
$$F(x_0,...,x_n) + y^d,$$
where $F$ is a form of degree $d$. In other terms, adding the power of a new variable increases the rank by exactly one.

Because of this remark, in the following theorem we may assume that the polynomial rings all have at least two variables.

\end {rem}

\begin{thm}\label{mainthm1}
Let $F = F_1+ \cdots + F_m \in S$, where $F_i \in S^{[i]}$ with $n_i \geq 1$. If all the forms $F_i$ are $e$-computable and $(F_i^\perp)_e=0$ then
\[ \mathrm{rk}(F) =  \mathrm{rk} (F_1 )+ \cdots + \mathrm{rk} (F_m ),\]
that is the Strassen Conjecture is true for $F$.
\end{thm}
\begin {proof}
Let $I^{[i]} \subset T^{[i]}$ and $t_i$  ($\deg t_i=e$) compute the rank of $F_i$ and let $V_i$ be the zero set of $I^{[i]}$.
It is enough to prove that
\[ \mathrm{rk}(F) \geq \mathrm{rk} (F_1 )+ \cdots + \mathrm{rk} (F_m ),\]
since the opposite inequality is obvious.

If $\X$ minimally decomposes $F$, then the ideal $I_\X : (I^{[1]} + \cdots +I^{[m]}) $ is the homogeneous ideal of the subset $\X ' $ of $\X$ not lying on $V_1 \cap \dots \cap V_m$.

For a general choice of $a_i\in k$, the form $ a_1t_1+ \cdots +a_mt_m$  is a non zero divisor for $I_{\X '}$.
Now consider $  I_{\X '} + (a_1t_1+ \cdots +a_mt_m) $.
We have
\[ I_{\X '} + (a_1t_1+ \cdots +a_mt_m)  \]
\[= (I_\X : (I^{[1]} + \cdots +I^{[m]})) +(a_1t_1+ \cdots +a_mt_m)
\]

\[\subseteq
(F^\perp : (I^{[1]}+ \dots+ I^{[m]}) )+(a_1t_1+ \cdots +a_mt_m).
\]
Hence, by Lemma \ref {lemma1ecomp},
\[  I_{\X '} + (a_1t_1+ \cdots +a_mt_m)   \subseteq  J_1 \cap \dots \cap J_m,
  \]
  where $J_i = (F_i^\perp :I^{[i]})+(t_i) \subset T $, considering $F_i\in S$.

We say that a degree $e$ form $h \in I_{\X '}$ is    {\it uniform}  if
$$h = h_1 + \ldots+  h_m ,$$
and   $h_i $ ($i = 1, \ldots, m$) is a degree $e$ form in $T^{[i]}$.

{\em Claim 1: If $h\in (I_{\X '})_e$ is uniform, then $h$=0.
%There are no uniform degree $e$ forms $h\neq 0$ in $I_{\X '}$.
}

Assume that $h\in (I_{\X '})_e$ is uniform.

  Since  $ I_{\X '} = I_{\X} : (I^{[1]}+ \cdots +I^{[m]})$, and $I_{\X} \subset F^\perp$, then $hl_i \in F^\perp$, for any $l^{[i]} \in I_i$.
Hence, for every $i=1,...,m$,
  $$hl_i \in F^\perp \Rightarrow hl_i \circ F =0 \Rightarrow hl_i \circ F_i =0 \Rightarrow h_il_i \circ F_i =0  .$$
 Now, considering $F_i \in S^{[i]}$, the last equality implies
 $h_i \in  F_i ^\perp :l_i$, so that $h_i \in (F_i ^\perp :I^{[i]} )$ and $h_i \in (F_i ^\perp :I^{[i]} )+(t_i) \subset T^{[i]}.$

  Hence, by Proposition \ref{prop1ecomp}, $h_i \in I_{  \X_i } +(t_i)$, where $\X_i$ minimally decomposes  $ F_i$. By hypothesis $(F_i^\perp)_e=0$, hence there are no degree $e$ forms in $I_{  \X_i }$. Thus we have
   $h_i = \mu_i t_i $, and
  $$h= \mu_1 t_1+ \ldots+   \mu_m t_m .$$

  Recall that $h \in I_{\X '}$ and hence it vanishes on all the points of $\X '$, that is the points of $\X$ not lying on $V_1 \cap \dots \cap V_m$.   Since $t_i \in I^{[i]}$, we have that $h$ vanishes also on
  $V_1 \cap \dots \cap V_m   $. It follows that $h \in I_\X \subset F^\perp$.
    Thus $ h \circ F=0$. Now
  \[ h \circ F= h \circ (F_1 + \cdots +F_m) =
\]
\[
 ( \mu_1 t_1+ \ldots+   \mu_m t_m) \circ (F_1 + \cdots +F_m)=
 \mu_1 t_1 \circ F_1 + \ldots+   \mu_m t_m \circ F_m.
 \]
Since $n_i \geq 1$ for all $i=1,...,m$, the hypothesis $(F_i)^\perp _e = 0 $ implies that $\deg F_i >e$, and hence $\deg t_i \circ F_i >0$.
 It follows that $\mu_i t_i \circ F_i =0$ for all $i=1,...,m$, that is $\mu_i t_i \in F_i^\perp$ (considering $F_i \in S^{[i]}$).
Since $(F_i)^\perp_e=0$, we get that $\mu_i=0$ for every $i$, and hence $h=0$.
This completes the proof of  Claim 1.

{\em Claim 2:  If $B$  be is basis of $(I_{\X '})_e $, then
$ B \cup \{t_1,...,t_m\} $ is a set of linearly independent forms.}

For $e=1$ Claim 2 follows immediately from Claim 1, so assume $e>1$.

Let
$$B=\{\alpha_1+\widetilde \alpha_1, ...,\alpha_l+ \widetilde \alpha_l \},$$
where the $\alpha_i$ are uniform and the $\widetilde \alpha_i$ are not uniform.
Now if $t_1$ (and analogously for $t_2,...,t_m$) satisfies:
$$t_1= \mu_1 (\alpha_1+\widetilde \alpha_1)+ \cdots +\mu_l(\alpha_l+\widetilde \alpha_l)+\nu_1 t_2 + \cdots + \nu_m t_m,$$
we get $\mu_1 \widetilde \alpha_1+ \cdots +\mu_l\widetilde \alpha_l=0$.
Hence
$$\mu_1 (\alpha_1+\widetilde \alpha_1)+ \cdots +\mu_l(\alpha_l+\widetilde \alpha_l) =
\mu_1 \alpha_1+ \cdots +\mu_l\alpha_l \in (I_{\X '})_e.
$$
Claim 1 yields
$\mu_1 \alpha_1+ \cdots +\mu_l\alpha_l =0$.
It follows that $t_1$ is a linear combination of $t_2,...,t_m$, thus a contradiction. This finishes the proof of  Claim 2.

Hence by Lemma \ref{lemma1ecomp} we have
$$I_{\X '} + (a_1t_1+ \cdots +a_mt_m) \subseteq J_1 \cap ... \cap J_m.
$$
 Since $B \cup \{a_1t_1+ \cdots +a_mt_m\}$ is a basis of $(I_{\X '} + (a_1t_1+ \cdots +a_mt_m) )_e$ and, by Claim 2,
 $B \cup \{t_1,...,t_m\} \subseteq J_1 \cap \dots \cap J_m$  is a set of linearly independent forms, then
 we have
 \[HF (T/ I_{\X '}  + (a_1t_1+ \cdots +a_mt_m),  e)-
 HF (T/ J_1 \cap \dots \cap J_m,  e) \geq m-1. \]

Since   $ a_1t_1+ \cdots +a_mt_m$  is a non zero divisor for $I_{\X '}$,
for $s \gg 0$ we have
\[ \mathrm{rk}( F )\geq |\X| \geq |\X '| = HF (T/  I_{\X '} ,s) =
\left({1\over e} \right) \sum _{i=0}^s HF (T/ I_{\X '}  + (a_1t_1+ \cdots +a_mt_m), i )
\]
\[\geq \left({1\over e} \right)  \left ( \sum _{i=0}^{e-1} HF (T/ J_1 \cap \dots \cap J_m,  i)+(HF (T/ J_1 \cap \dots \cap J_m,  e) +m-1)+ \right.\]
\[\left. + \sum _{i=e+1}^s HF (T/ J_1 \cap \dots \cap J_m ,i)\right ).
\]
Hence, for $s \gg 0$,  by Lemma \ref{lemma2ecomp}, we get
\[\mathrm{rk}( F )\geq     \left({1\over e} \right)  \left ( \sum_{i=0}^s HF(T/  J_1\cap \ldots \cap J_m ,i)+m-1\right )=
\mathrm{rk} (F_1)+ \cdots + \mathrm{rk} (F_m).
\]
  \end {proof}

\section{Forms for which the Strassen Conjecture holds} \label {strassen}

\begin {thm}\label{coroll}
Let $F = F_1+ \cdots + F_m \in S_d$, where $F_i \in S^{[i]}_d$.    If, for $i=1,\ldots,m$, $F_i$ is of one of the following types:
\begin {itemize}
\item $F_i$ is a monomial;
\item $F_i$ is a form in one or two variables;
\item $F_i = x_0^a(x_1^b+ \cdots +x_n^b)$ with $a+1 \geq b$;
\item $F_i = x_0^a(x_1^b+ x_2^b)$;
\item $F_i = x_0^a(x_0^b+x_1^a+ \cdots +x_n^b)$ with $a+1 \geq b$;
\item $F_i = x_0^a(x_0^b+x_1^b+ x_2^b)$;
\item $F_i=x_0^aG(x_1,\ldots,x_n)$ where $G^\perp=(g_1,\ldots,g_n)$ is a complete intersection and $a<\deg(g_i)$ for $i=1,\ldots,n$;
\item $F_i$ is a Vandermonde determinant;
\end {itemize}

then the Strassen Conjecture holds for $F$.

\end {thm}
\begin {proof} All the forms above are $1$-computable, hence the conclusion follows from
Remark \ref{rem2}, Remark \ref {rem3}, Theorem \ref {mainthm1} with $e=1$, and in the case of Vandermonde determinant, Remark \ref {remVand},
\end{proof}

\begin {rem}
If $F$ is a form which is $e$-computable, but not $1$-computable, we can only combine it with other $e$-computable forms to get a form satisfying Strassen's Conjecture.

For example, if $F$ is the form of Example \ref {example2},
then we know that $F$ is $2$-computable and $\mathrm{rk}(F)=25$,
but we know   $F$ is not $1$-computable by Example \ref{example1}.

If $G_1= x_0x_1^4x_2^5$ then we showed that $G_1$ is $1$-computable and $\mathrm{rk}(G_1)=30$.
But we do not know if $G_1$ is $2$-computable.

Thus we cannot use the theorem to find the rank of  $F+G_1$, although Strassen's conjecture says that the rank should be $25 + 30$.

However, if $G_2= x_0^3 x_1^4x_2^5$, by Proposition \ref{monomi},  we know that $G_2$ is $2$-computable and $\mathrm{rk}(G_2)=30$. Hence
$$\rk (F+G_2) = 25+30=55.$$
\end {rem}

%%%%%%%%%%%%%%%%

\begin {rem}
{It would be interesting to have a characterization  of those $F\in k[x_0,x_1]$ for which $F^\perp = (q^a,h_2)$ with $a \geq 2$.
If we had that, we would have examples which were $\deg q$-computable. This would give us more forms for which Strassen's conjecture is true.}

\end  {rem}

%%%%%%%%%%%%%%%%%%%%%%%%%%%%%%%%

\section{Some examples} \label {examples}
%%%%%%%%%%%%%%%%%%%%%%%%%%%%%%%%

\begin{lem}\label{complement} Let $F = x_0^a(x_1^b+\cdots+x_n^b)$ with $a+1 \leq b$, $n \geq3$. If $\mathbb X $ is apolar to $F$, then $|\mathbb X \setminus \{X_i=0\}| \geq b$ for all $i=1,...,n$.
\end{lem}

\begin{proof}
Since $I_{\mathbb X} : (X_i) \subseteq F^{\perp} : (X_i) = (x_0^ax_i^{b-1})^{\perp}$  and $\rk(x_0^ax_i^{b-1}) =b$  (see \cite {carcatgermonomi}), the Apolarity Lemma, yields that the ideal $I_{\mathbb X} : (X_i)$ is the homogeneous ideal of a  set of at least $b$ points. That is, $|\mathbb X \setminus \{X_i=0\}| \geq b$ for all $i=1,...,n$.
\end{proof}

\begin{prop}\label{3bpunti} If $F = x_0^a(x_1^b+\cdots+x_n^b)$ with $2 \leq a+1 \leq b$ and $n \geq3$, then $$ bn-n+3 \leq \rk(F) \leq bn.$$
In particular, we have $\rk(x_0^a(x_1^b+x_2^b+x_3^b)) =3b$.
\end{prop}

\begin{proof}  Note that $F^{\perp} = (X_0^{a+1},X_1X_2,X_1X_3,...,X_{n-1}X_{n},X_1^b-X_2^b,...,X_1^b-X_n^b)$. \\
We split the proof into four steps.

Step 1: $\rk(F) \leq bn$\\
It is easy to see that $I = (X_1X_2,X_1X_3,...,X_{n-1}X_{n},(n-1)X_1^b-X_2^b-\cdots-X_n^b-X_0^b) \subseteq F^{\perp}$  is the homogenous ideal of a set of $bn$ distinct  points. By the Apolarity Lemma $\rk(F) \leq bn$ .

Step 2:  $bn-n+2 \leq \rk(F)$\\
Let $\tilde{I} = F^{\perp} : (X_0) + (X_0)=(X_0,X_1X_2,X_1X_3,...,X_{n-1}X_{n},X_1^b-X_2^b,...,X_1^b-X_n^b)$. Thus we have
\[
 \begin{array}{c|cccccc}
   i                   &  0& 1 & \cdots & b-1 & b & b+1\\
  \hline \\
  HF(T/\tilde{I},i)   &  1&n & \cdots & n    &1 &0 \\

 \end{array}
 .\]
Hence, by Corollary \ref {maincor}, we get $\rk(F) \geq \sum_{i \geq 0} HF(T/\tilde{I},i) = bn-n+2$.

Step 3: {\it  Let $ \mathbb X$ be apolar to  $F$ and $a=1$. If $X_iX_j +c_{ij} X_0^2 \in I_{\mathbb X}$  for all $1\leq i<j \leq n$, then $c_{ij}=0$ for all $i,j$.}

Suppose that $c_{ij} \neq 0$ for some $i<j$. Say, $c_{12} \neq 0$, then we have $X_1X_2 +c_{12}X_0^2, \ X_1X_3+c_{13}X_0^2 \in I_{\mathbb X}$. Thus $X_1(c_{13}X_2-c_{12}X_3)  \in I_{\mathbb X}.$
Thus we have
\[X_1  \in (I_{\mathbb X} : (c_{13}X_2-c_{12}X_3))\]
and hence
\[c_{12}X_0^2 \in (I_{\mathbb X} : (c_{13}X_2-c_{12}X_3)).\]
Since the ideal is radical we get
\[X_0  \in  (I_{\mathbb X} : (c_{13}X_2-c_{12}X_3))\]
and thus
\[X_0(c_{13}X_2-c_{12}X_3) \in I_{\mathbb X}\]
and this yields the contradiction $c_{12} = 0$ and $c_{13} = 0$.

Step 4: $bn-n+2 < \rk(F)$.\\
Suppose that $\rk(F)=bn-n+2=|\mathbb X| $, where $\X$  minimally decomposes $F$.

By  the proof of Step 2, the rank of F is computed by  $X_0$, hence
by Proposition \ref{prop1ecomp}  we get  $I_\X+ (X_0)= F^\perp+(X_0)$.
 In particular we have $X_iX_j \in I_{\mathbb X}+(X_0)$ for all $1 \leq i<j \leq n$, and so $X_iX_j + L_{ij}X_0 \in I_{\mathbb X}$ for some linear form $L_{ij}$. Since $I_{\mathbb X} \subset F^{\perp}$ and $X_iX_j \in F^{\perp}$, then $L_{ij}X_0 \in F^{\perp}$.

 If $a>1$, then $L_{ij}=0$.

 Let $a=1$. We get $L_{ij}= c_{ij}X_0$ and hence $X_iX_j +c_{ij} X_0^2 \in I_{\mathbb X}$. By Step 3, we have $c_{ij}=0$.

 Consequently, $X_iX_j \in  I_{\mathbb X}$ for all $1\leq i<j\leq n$ and for any $a \geq 1$.\\
 Now, since the ideal $(X_1X_2,X_1X_3,...,X_{{n-1}{n}})$ is the homogeneous ideal  of $n$ lines $l_1,...,l_n$ where $ l_i = \{X_1=X_2=\cdots=\hat{X_i}=\cdots=X_n=0\}$, it follows that all the points of $\X$ lie on the union of the lines  $l_i$. Since  $ \X \setminus \{X_i = 0\} = \X \cap (l_i \setminus (1,0,...,0))$, by Lemma \ref{complement} we have that
\[bn-n+2= | \X| \geq  \sum_{i=1} ^n |\mathbb X \setminus \{X_i = 0\}| \geq bn ,\]
 a contradiction.
 \end{proof}

\begin {rem} The form $F = w(x^3+y^3+z^3)\in k[x,y,z,w]$  is not 1-computable.

 If  $F$ is 1-computable, then there exists an ideal
$I \subset T$  of a linear space $L$  such  that

 \[ \rk (F) =
\sum_{i=0}^\infty HF(T/(F^\perp: I +(t)),i) ,\]
where  $t=aX+bY+cZ+dW  \in I$ is a general linear form.

If $t$ has at least two of the coefficients  $a,b,c,d$ different from zero, since
$$
F^\perp= (W^2, YZ, XZ, XY,  Y^3 -Z^3, X^3 -Z^3),
$$
we get that
$$HF(T/(F^\perp +(t)),0) =1$$
$$HF(T/(F^\perp +(t)),1) =3$$
$$HF(T/(F^\perp +(t)),2) \leq 3$$
$$HF(T/(F^\perp +(t)),3) \leq 1$$
$$HF(T/(F^\perp +(t)),4) =0.$$
By Proposition \ref{3bpunti}  we know that $\rk (F) =9$ and since
$$
F^\perp +(t) \subseteq F^\perp: I +(t)
$$
we get
$$ 8 \geq \sum_{i=0}^\infty HF(T/(F^\perp+(t)),i)  \geq
\sum_{i=0}^\infty HF(T/(F^\perp: I+(t)),i) =\rk (F) =9
,$$
and this is a contradiction.

Now if $L$ is a point or a line, and  $\{t=0\}$ is a general plane through $L$, then $t$ has at least two of the coefficients  $a,b,c,d$ different from zero. If $L$ is a plane, then $ (t)=I $, and the only planes with three coefficients zero between $a,b,c,d$ are the coordinate planes.
Hence the only possibility for  $F $ to be 1-computable, is with $L= \{X=0\}, \{Y=0\}, \{Z=0\},\{W=0\}$, but
$$\sum_{i=0}^\infty HF(T/(F^\perp: (X) +(X)),i) =2,
$$
(analogously for $Y$ and $Z$) and
$$\sum_{i=0}^\infty HF(T/(F^\perp: (W) +(W)),i) =8.
$$

Hence, $F$ is not $1$-computable.

\end {rem}

%%%%%%%%%%%%%%%%%%%%%%%%%%%%%%%%%%%%%%%

%%%%%%%%%%%%%%%%%%%%%%%%%%%%%%%%%%%%%%%%

%\bibliographystyle{alpha}
%\bibliography{biblioCarlini}

%%%%%%%%%%%%%%%%%%%%%%%%%%%%%%%%%

\end{document}